\documentclass[artcle, 11pt, reqno]{amsart}

\usepackage{lipsum}
\usepackage{amsfonts}
\usepackage{graphicx}
\usepackage{epstopdf}
\usepackage{algorithmic}
\usepackage[utf8]{inputenc}
\usepackage{mathtools}
\usepackage{amssymb}
\usepackage{hyperref}
\usepackage{textcomp}
\usepackage[english]{babel}
\usepackage{verbatim}
\usepackage{todonotes}

\usepackage{bbm}
\usepackage{float}
\usepackage{mathrsfs}
\numberwithin{equation}{section} 

\newcommand{\R}{\mathbb{R}}

\newcommand{\N}{\mathbb{N}}
\newcommand{\Hd}{\mathcal{H}}
\newcommand{\pom}{{\partial \Omega}}
\newcommand{\ind}{\mathbbm{1}}
\newcommand{\C}{\mathcal{C}}
\newcommand{\contvf}{C^1_c(\R^n; \R^n)}
\newcommand{\ninfty}{n \rightarrow \infty}
\newcommand{\dd}{\: d}
\newcommand{\Enumerator}{\int_{\partial \C} \ind_E(x) \dd \Hd^{n-1}}
\newcommand{\divv}{\textnormal{div} \;}
\DeclareMathOperator{\Tr}{Tr}

\ifpdf
\DeclareGraphicsExtensions{.eps,.pdf,.png,.jpg}
\else
\DeclareGraphicsExtensions{.eps}
\fi

\theoremstyle{plain}
\begingroup
\newtheorem{theorem}{Theorem}
\newtheorem{lemma}[theorem]{Lemma}
\newtheorem{proposition}[theorem]{Proposition}
\newtheorem{corollary}[theorem]{Corollary}
\endgroup
\theoremstyle{definition}
\begingroup
\newtheorem{definition}[theorem]{Definition}
\newtheorem{remark}[theorem]{Remark}

\endgroup
\theoremstyle{remark}

\makeatletter
\def\namedlabel#1#2{\begingroup
	#2%
	\def\@currentlabel{#2}%
	\phantomsection\label{#1}\endgroup
}
\makeatother

\title{Optimal constant for the trace inequality in {$BV$} for domains with corners}

\author{Riccardo Cristoferi}
\author{Devin van der Gulik}

\usepackage{amsopn}

\ifpdf
\hypersetup{
	pdftitle={Optimal constant for the trace inequality in {$BV$} for domains with corners},
	pdfauthor={}
}
\fi

\begin{document}
	
	\maketitle
	
	\begin{abstract}
		We determine the explicit value of the optimal constant in the trace inequality for functions of bounded variations in the case the domain has a particular class of singularities.
	\end{abstract}
	
	\section{Introduction}

	Derivatives and integrals are inherently connected through the fundamental theorem of calculus, with its multi-dimensional equivalent being the Gauss-Green Theorem, also known as the Divergence Theorem. This fundamental result states that, given a set $\Omega \subset \R^n$ with piecewise smooth boundary, and $T \in C^1(\R^n;\R^n)$, it holds that 
	\begin{equation} \label{eq:GGtheorem}
		\int_E \textnormal{div } T\dd x = \int_{\partial E} T \cdot \nu_E \dd\Hd^{n-1},
	\end{equation}
	where with $\nu_\Omega(x)$ we denote the outward normal to $\Omega$ at $x \in \partial \Omega$, and with $\Hd^{n-1}$ the surface measure on $\partial\Omega$ (or, more precisely, the $(n-1)$-dimensional Hausdorff measure, see Definition \ref{DefHausdM}).
	In particular, if $u\in C^1(\overline{\Omega})$, then the above results yields that
	\begin{equation}\label{eq:GG_reg}
		\int_\Omega \partial_i u(x)\dd x = \int_{\partial\Omega} u(x)\nu_i(x)\dd\Hd^{n-1},
	\end{equation}
	for all $i=1,\dots,n$, where $\nu_i(x)$ denotes the $i^{th}$ component of the vector $\nu(x)$.
	Note that this theorem requires strong regularity on $u$ in order to give meaning of $u(x)$ for $x\in\partial\Omega$. However, many of the subjects in applied analysis work with more general function spaces than $C^1$, that use a weaker notion of derivatives, such as Sobolev spaces, or functions of bounded variation (see \cite{leoni2017first}).
	When one tries to generalize \eqref{eq:GG_reg} to functions belonging to these spaces, one of the problems to tackle is how to give a meaning to $u(x)$ for $x\in\partial\Omega$.
	Fortunately, it is possible to solve such a problem by using the notion of \emph{trace}. We now focus on the case of functions of bounded variations in $\Omega$, that we denote by $BV(\Omega)$. If $\Omega$ has a Lipschitz boundary (which will be assumed for the rest of this section), there exists a unique bounded linear operator
	\[
	\Tr: BV(\Omega) \rightarrow L^1(\pom)
	\]
	such that, for $u \in BV(\Omega)$, we have  
	\[
	\int_\Omega \partial_i u(x)\dd x = \int_{\partial\Omega} \Tr(u)(x)\nu_i(x)\dd\Hd^{n-1},
	\]
	for all $i=1,\dots,n$. Moreover, if $u\in BV(\Omega)\cap C^0(\overline{\Omega})$, then $\Tr(u)(x)=u(x)$ for all $x\in\partial\Omega$.
	
	Boundedness of the operator $\Tr$ means that there exists a constant $C>0$ that depends only on $\Omega$ and on the dimension $n$ such that
	\[
	\int_{\partial\Omega} |\Tr(u)(x)|\dd\Hd^{n-1} \leq C \|u\|_{BV(\Omega)}
	= C\left[ \int_\Omega |u(x)|\dd x + |Du|(\Omega) \right],
	\]
	for all $u \in BV(\Omega)$. Here, the term $|Du|(\Omega)$ denotes the total variation of the measure $Du$ in $\Omega$. For the reader not used to this notion, it can be seen as a generalization of $\|\nabla u\|_{L^1(\Omega;\R^n)}$.
	Nevertheless, this estimate is, in general, not enough when dealing with problems in the Calculus of Variations where both boundary and bulk energies are involved (see, for instance \cite{Modica, cristoferi2020relaxation}). As an example, consider the functional $\mathcal{F}: BV(\Omega)\to\R$ defined as
	\[
	\mathcal{F}(u) := |Du|(\Omega) + \alpha \int_{\partial\Omega} |\Tr(u)(x)|\dd\Hd^{n-1} ,
	\]
	for some $\alpha\in\R$. We want to show that the minimization problem
	\[
	\min_{u\in BV(\Omega)} \mathcal{F}(u)
	\]
	has a solution.
	Note that if $\alpha\geq0$, then the problem is trivial, since it is solved by the function $u\equiv0$. On the other hand, when $\alpha<0$, we need to ensure that
	\[
	\int_{\partial\Omega} |\Tr(u)(x)|\dd\Hd^{n-1} \leq -\frac{1}{\alpha}|Du|(\Omega),
	\]
	otherwise the minimization problem doesn't admit a solution, since the infimum of $\mathcal{F}$ over $BV(\Omega)$ would be $-\infty$. What is the lower possible value of $\alpha$ for which the problem admits a solution? The answer to this question is equivalent to find the lower possible constant $C>0$ such that there exists $\lambda>0$ such that
	\[
	\int_{\partial\Omega} |\Tr(u)(x)|\dd\Hd^{n-1}
	\leq C\int_\Omega |u(x)|\dd x + \lambda |Du|(\Omega),
	\]
	for all $u\in BV(\Omega)$.
	The validity of such an estimate has been established by Anzellotti and Giaquinta in \cite{anzellotti1978funzioni}.
	In particular, they proved that there exists a constant $Q_{\partial\Omega}>0$ with the following property: for every $r > 0$ there exists $C(r, \Omega) > 0$ such that
	\begin{equation} \label{eq:trineq}
		\int_{\partial^\ast \Omega} |\Tr(u)(x)|\dd\Hd^{n-1} \leq (Q_\pom + r)|D u|(\Omega) + C(r, \Omega)\int_\Omega |u(x)|\dd x,
	\end{equation}
	for all $u \in BV(\Omega, \R^M)$.
	They also provided a characterization of such a constant $Q_\pom$, and proved that it is the infimum among all the constants $Q$ for which there exists $C$ with the property that a bound of the form \eqref{eq:trineq} holds for every $u \in C^1_c(\Omega)$.
	It holds that
	\[
	Q_\pom \coloneqq \sup \{q_\pom(x) : x \in \pom\},
	\]
	where
	\begin{equation} \label{Def:qpom}
		q_\pom(x) \coloneqq \lim_{\rho \to 0^+} \sup\Bigl\{ \frac{1}{P(E;\Omega)} \int_\pom \Tr(\ind_E)(x) \dd\Hd^{n-1} : E \subset B(x,\rho), |E| >0, P(E;\Omega) < \infty \Bigr\},
	\end{equation}
	where $P(E;\Omega)$ denotes the perimeter of $E$ in $\Omega$ (see Definition \ref{def:perimeter}).
	Thus, to answer our questions, we now have to solve an asymptotic geometric variational problem. In the case that $\pom$ is locally $C^1$ around a point $x \in \pom$, Giusti proved in \cite{giusti1976boundary} that $q_\pom(x) = 1$ (see Theorem \ref{thm:Giusti}).
	Moreover, it holds that $q_\pom(x) \geq 1$. As visualized in Figure \ref{fig:qpom}, $q_\pom$ can attain higher values than 1 whenever $\pom$ has corners that are outward-pointing with respect to $\Omega$. Very little is know for the explicit value of $q_{\partial\Omega}$ for a general set $\Omega\subset\R^n$.
	
	\begin{figure}[h] \label{fig:qpom}
		\centering  \includegraphics[scale=0.5]  {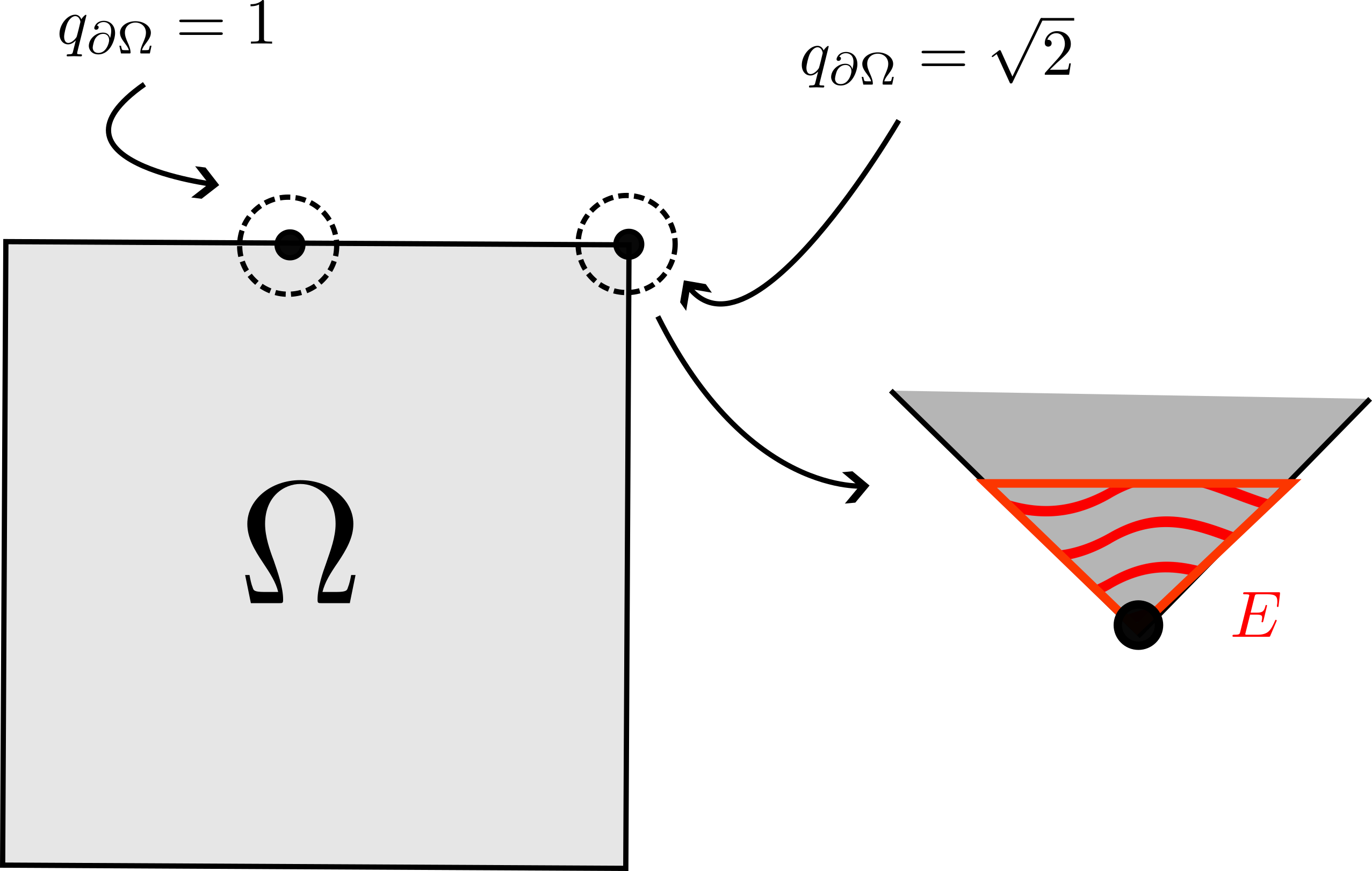}
		\caption{Example of a domain where $q_\pom$ varies between 1 and $\sqrt{2}$ on $\pom$, so that $Q_\pom = \sqrt{2}$.}
		\label{fig:enter-label}
	\end{figure}
	
	The aim of this paper is to consider the general case $n\geq 2$, and obtain the explicit value of \eqref{eq:trineq} for points in $\partial\Omega$ with a specific type of singularity, that generalised that of angles in dimension $n=2$.

	\subsection{Main result}
	
	What is left open in the theory, is to understand what happens when at a point $x\in\partial\Omega$, the boundary is not of class $C^1$. Here, we focus on a particular type of singularity that generalises that of an angle in dimension $n=2$. Without loss of generality, we take the point $x$ to be the origin.
	
	\begin{definition} \label{def:cone}
		Let $\C\subseteq \R^n$ be an open subset. We say that $\C$ is a \textbf{cone} if there exists a continuous function $f: \mathbb{S}^{n-2} \rightarrow (0, \infty]$ such that
		\[
		\C = \biggl\{ x \in \R^n : \Big| \frac{x'}{x_n} \Big|<  f \left( \frac{x'}{x_n} \right) \biggr\},
		\]
		where, for each $x\in\R^N$, we write \( x = (x', x_n)\), with $x'\in\R^{n-1}$, and $x_n\in\R$.
		We call \textbf{basis} of the cone, the set $G(f) \subset \R^{n-1}$ defined $\{x'\in\R^{n-1} : (x',1)\in\C\}$.
	\end{definition}
	Thanks to the homogeneity of the cone, we get that
	\[
	q_\pom(0) = \sup \Bigl\{ \frac{1}{P(E;\C)} \int_{\partial \C} \Tr(\ind_E)(x) \dd \Hd^{n-1} : |E| > 0, E \subseteq \C \text{ of finite perimeter} \Bigr\}.
	\]
	Before stating the main result of this paper, we give some heuristics. Consider, in dimension $n=2$, an angle like in Figure \ref{fig:qpom} on the right. The solution to the minimization problem defining $q_\pom(x)$ in this case is given by a segment which is orthogonal to the bisectrix of the angle. We expect the same to hold also in higher dimension. The main result of this paper confirms this heuristics in the case of cones having an inner tangent sphere.
	
	\begin{theorem}
		Let $\C$ be a cone with a base $G(f) \subset \R^{n-1}$. Suppose that there exists a sphere 
		\[
		S_R(\bar{x}) \coloneqq \{ y \in \R^{n-1} : \|\bar{x}-y\| = R \} \subset \R^{n-1},
		\]
		such that any hyperplane $T_y \subset \R^{n-2}$ tangent to $G(f)$ at $y \in G(f)$ is also tangent to $S_R(x)$ at the same point. Then, we have 
		\[
		q_\pom(0)=\frac{\sqrt{1+R^2}}{R}.
		\]
		Moreover, a set $E\subset\C$ is a solution to the maximization problem defining $q_\pom(0)$ if and only if it is of the form $E=\C\cap \{x\in\R^n : x\cdot \bar{x} \leq t\}$ for some $t>0$.
	\end{theorem}

	\section{Preliminaries} \label{sec:prelim}
	In this section we introduce the necessary mathematical language necessary to define the perimeter of a set and mention its relevant properties, by using some tools from geometric measure theory.
	We start by introducing the notion of trace of a set on a regular boundary.
	
	\begin{definition} \label{def:trace}
		Let $\Omega\subset\R^n$ be an open set with Lipschitz boundary.
		For a set $E \subset \R^n$ and $x \in \R^n$ the \textbf{Trace} of $E$ on $\partial \Omega$ is defined as
		\[
		\Tr(\ind_E)(x) =  \lim_{\rho \rightarrow 0^+} \frac{|E \cap B(x, \rho) \cap \Omega|}{|B(x, \rho) \cap \Omega|},
		\]
		for all $x\in\partial\Omega$.
	\end{definition}

	\subsection{Hausdorff measure}
	As we will be dealing with boundaries of sets, we need a way to determine whether it is well-defined and of finite area. To this end, we first introduce the Hausdorff measure:
	
	\begin{definition}
		Let $n \geq k>0$. Then for $\delta > 0$, the \emph{outer measure $\mathcal{H}^k_{\delta}$ of step $\delta$} of a set $E \subset \R^n$ is defined as 
		\[
		\mathcal{H}^k_{\delta}(E) \coloneqq \inf_{\mathcal{F}} \sum_{F \in \mathcal{F}} \omega_k \left( \frac{\textnormal{diam}(F)}{2} \right) ^k.
		\]
		Here, $\mathcal{F}$ is any covering of $E$ such that for any $F \in \mathcal{F}$ we have $ \textnormal{diam}(F)<\delta$, where diam$(F)$ denotes the diameter of the set.
	\end{definition}
	Note that $\mathcal{H}^k_\delta$ is not a measure, but it does define an outer measure on $\R^n$. However, taking $\delta$ to 0 we do find a well-defined measure.
	
	\begin{definition} \label{DefHausdM}
		Let $n \geq k>0$. Then, the \textbf{$k$-dimensional Hausdorff measure} of a set $E \subset \R^n$ is defined as
		\[
		\mathcal{H}^k(E) \coloneqq \sup_{\delta \in (0, \infty]} \mathcal{H}^k_\delta(E) = \lim_{\delta \to 0^+} \mathcal{H}^k_\delta(E).
		\]
	\end{definition}
	
	
	\subsection{Sets of finite perimeter}
	In this section we will now show how we can use Riesz´s theorem to construct a vector-valued surface measure, giving meaning to the notion of the perimeter of sets, and describe how it is related to the Hausdorff measure.
	
	\subsubsection{The Gauss-Green measure} \label{subsec:GGmeasure}
	We have a look at equation \eqref{eq:GGtheorem}. If $E$ has $C^1$-boundary, and there exists $T\in C^1(\R^n;\R^n)$ with $|T|=1$, such that $T(x) = \nu_E(x)$ for all $x \in \partial E$, then
	\begin{equation} \label{eq:T_C1bndry}
		\int_E \textnormal{div } T\dd x
		= \int_{\partial E} T \cdot \nu_E\dd\Hd^{n-1}
		= \int_{\partial E} 1 \: \dd\Hd^{n-1}
		= \Hd^{n-1} (\partial E).
	\end{equation}
	Note that the  right-hand side is also the greatest value attained by the functional
	\[
	T\mapsto \int_E \textnormal{div } T
	\]
	over vector fields $T\in C^1(\R^n;\R^n)$ with $\sup_{\R^n} |T| \leq 1$.
	This motivates the following definition.
	
	\begin{definition} \label{def:perimeter}
		Let $A\subset\R^n$ be an open set. A set $E \subseteq \R^n$ is a \textbf{set of locally finite perimeter in $A$} if for all $K \subset A$ compact we have 
		\begin{equation} \label{eq:LocFin}
			\sup\left\{ \int_E \textnormal{div } T \: : \: T \in C^1_c(A;\R^n), \: \textnormal{spt }T \subseteq K, \: |T| \leq 1 \right\} < \infty.
		\end{equation}
		Here, $|T| \leq 1$ is understood as $\sup_{A} |T(x)| \leq 1$.
		In particular, we say that $E$ is a \textbf{set of finite perimeter in $A$} when the above bound is independent of $K$.
	\end{definition}
	If we now look at equation \ref{eq:T_C1bndry} again, one may wonder whether it is possible to find some vector valued $\mu_E$ for sets of locally finite perimeter $E$ so that integrating $T$ with respect to $\mu_E$ leads to a generalized divergence theorem.
	Indeed, noting that the integral term in Definition \ref{def:perimeter} is a linear functional over $C^1_c(\R^n;\R^n)$, we can apply Riesz's representation theorem to represent it as a measure.
	
	\begin{proposition}
		Let $E \subseteq A$ be a Lebesgue measurable set. Then, $E$ is a set of locally finite perimeter in $A$ if and only if there exists an $\R^n$-valued Radon measure $\mu_E$ on $A$, called the \textnormal{\textbf{Gauss-Green measure}}, such that
		\begin{equation} \label{eq:GGmeasure}
			\int_E \textnormal{div }T = \int_{A} T \cdot \textnormal{d}\mu_E,
		\end{equation}
		for all $T \in C^1_c(A;\R^n)$.
	\end{proposition}
	From this vector-valued measure $\mu_E$, we can now construct a real-valued non-negative measure by taking its total variation $|\mu_E|$.
	For a set of locally finite perimeter $E$ we can now define its perimeter as follows.
	
	\begin{definition} \label{def:localperimeter}
		Let $E \subset \R^n$ be of locally finite perimeter in $A$. We define the $\textbf{relative perimeter}$ of $E$ in $F \subset A$ as 
		\[
		P(E;F) \coloneqq |\mu_E|(F).
		\]
		In case $A=F=\R^n$, we simply write $P(E)$.
	\end{definition} 
	We now investigate the relation between the Gauss-Green measure and the Hausdorff measure.
	Note that if $E$ has $C^1$-boundary, then
	\[
	\int_{\partial E} T \cdot \nu_E =
	\int_E \textnormal{div }T = \int_{\R^n} T \cdot \textnormal{d}\mu_E,
	\]
	for all $T \in C^1_c(\R^n;\R^n)$. This suggests a relation between $\mu_E$ and the point-wise defined unit normal $\nu_E$ to $\partial E$. This motivates the following definition.
	
	\begin{definition} \label{def:reducedbdry}
		Let $E \subset \R^n$ be a set of locally finite perimeter, then we define its \\ $\textbf{reduced boundary}$ to be 
		\[
		\partial^\ast E \coloneqq \Bigl\{x \in \textnormal{spt } \mu_E \: : \: 
		\lim_{r \rightarrow0^+} \frac{\mu_E(B(x,r))}{|\mu_E|(B(x,r))} \textnormal{ exists and lays inside } \mathbb{S}^{n-1} \Bigr\}.
		\]
	\end{definition}
	We define the $\textbf{measure-theoretic outer unit normal}$ $\nu_E$ to $E$ to be 
	\[
	\nu_E(x) = \lim_{r \rightarrow0^+} \frac{\mu_E(B(x,r))}{|\mu_E|(B(x,r))}. 
	\]
	
	\begin{remark}
		As both $\mu_E$ and $|\mu_E|$ are Radon measures (see [Chapter 2]\cite{maggi2012sets}) with the same support, applying the Lebesgue-Besicovitch differentiation theorem \cite[Theorem 5.8]{maggi2012sets} it follows that $\nu_E$ is defined $|\mu_E|$-a.e. and also
		\[
		\mu_E = \nu_E |\mu_E|\llcorner \partial^\ast E.
		\] 
		Consequently, we have
		\begin{equation} \label{eq:GeneralizedDivThm}
			\int_E \divv T = \int_{\partial^\ast E} T \, \nu_E \, \dd|\mu_E|,
		\end{equation}
		for all $T \in C^1_c(\R^n;\R^n)$.
	\end{remark}
	Furthermore, we can relate the Gauss-Green measure to the Hausdorff measure with the following theorem. The proof can be found in \cite[Theorem 15.9]{maggi2012sets}.
	
	\begin{theorem}[De Giorgi's structure theorem] \label{thm:Giorgo}
		If a set $E$ of locally finite perimeter in $\R^n$, then the Gauss-Green measure $\mu_E$ of $E$ satisfies
		\[
		\mu_E = \nu_E \Hd^{n-1} \llcorner \partial^\ast E, \quad |\mu_E| = \Hd^{n-1} \llcorner \partial^\ast E.
		\]
		Moreover, the generalized Gauss-Green formula holds true:
		\[
		\int_E \divv T \coloneqq \int_{\partial^\ast E} T \, \nu_E \, \dd \Hd^{n-1},
		\]
		for all $T \in C^1_c(\R^n;\R^n)$.
	\end{theorem}

	\subsubsection{Lower semi-continuity of perimeter and compactness} \label{subsec:Lsc}
	The standard approach to showing the existence of a minimizer of a minimal surface problem is to construct a minimizing sequence of sets and using  the properties of the Gauss-Green measure to prove that it will converge to some admissible set. \\
	
	Although we will later see that this method fails when it comes to our optimization problem, it is still relevant to see how it works and why exactly it will fail in our case. We first need a notion of convergence to describe what it means for a sequence of sets $E_n$ to converge to a certain limit $E$ in $\R^n$. This means that we want to quantify the difference between $E_n$ and $E$ and require it to tend to zero as $n \rightarrow\infty$. A logical choice is then to require the the volume of their symmetric difference to vanish. Since we want to describe this in general for sets of locally finite perimeter, we need the convergence to be described locally too, leading to the following definition.
	\begin{definition}
		For Lebesgue measurable sets $(E_n)_{n \in \N}$ and E in $\R^n$, we say that $E_n$ \textbf{locally converges} to $E$ and write $E_n \xrightarrow{loc} E$, if
		\[
		\lim_{n \rightarrow\infty} |K \cap (E \Delta E_n)| = 0 
		\]
		for all $K \subset \R^n$ compact. 
	\end{definition}
	It turns out that local convergence, together with a well-behaving perimeter, is sufficient for the local perimeter to be lower semi-continuous, that is
	\begin{proposition} \label{lscProp}
		If $(E_n)_{n \in \N}$ is a sequence of sets of locally finite perimeter in $\R^n$, with
		\begin{equation} \label{limsupcpt}
			E_n \xrightarrow{loc} E, \qquad \limsup_{n \rightarrow\infty} P(E_n; K) < \infty
		\end{equation}
		for all $K \subset \R^n$ compact, then $E$ is of locally finite perimeter in $\R^n$ and $\mu_{E_h} \xrightharpoonup{\ast} \mu_E$.\\ Moreover, for every open set $A \subset \R^n$ we have 
		\[
		P(E;A) \leq \liminf_{\ninfty} P(E_n;A).
		\]
	\end{proposition}
	With this proposition  we have established a connection between the local convergence of sets, and the convergence of their perimeters. 
	This proposition can now be used to prove the following theorem, which provides sufficient conditions for the existence of a solution to minimization problems involving the perimeter.
	Indeed, with Theorem \ref{ExistTh} at our disposal the only thing left in order to prove existence is showing that there is some minimizing sequence that is uniformly bounded.
	
	\begin{theorem}[Compactness] \label{ExistTh}
		If $R>0$ and $(E_n)_{n \in \N}$ are sets of finite perimeter in $\R^n$, with
		\begin{equation} \label{eq:ExistTh_SupCond}
			\sup_{\ninfty} P(E_n) < \infty
		\end{equation}
		\begin{equation} \label{eq:ExistTh_BallCond}
			E_n \subset B_R \quad \textnormal{for all } n \in \N
		\end{equation}
		then there exists a set $E \subset B_R$ and a subsequence $(E_{h_n})_n$
		\[
		E_{h(n)} \rightarrow E, \quad \mu_{E_{h(n)}} \xrightharpoonup{\ast} \mu_E.
		\]
	\end{theorem}
	For the proof we refer to \cite[Theorem 12.26]{maggi2012sets}.
	
	\subsubsection{Characterization of half-spaces} \label{subsec:halfspace}
	The following result states that whenever the normal of a set of locally finite perimeter is the same almost everywhere, then that set must equal a half-space. We will need this proposition to characterize solutions to the maximization problem later on. Since it is of great importance to our strategy, for the reader's convenience we will provide the proof as well, taken from \cite{maggi2012sets}.
	
	\begin{proposition}
		\label{halfspace}
		If $F$ is a set of locally finite perimeter in $\R^n$ and $\nu$ is such that $\nu_F(y)=v$ for $|\mu_F|$-a.e. $y \in \partial^\ast F$, then there exists $a \in \R^n$ so that we have 
		\[
		F = \Bigl\{ z \in \R^n : z \cdot v < a \Bigr\}.
		\]
	\end{proposition}
	
	\begin{proof}[Proof of Proposition \ref{halfspace}]
		In order to proof this Proposition we approximate $F$ with smooth surfaces, after which we show that their properties are preserved in the limit. To this is end, we introduce a regularizing kernel.
		Let $\rho \in C^\infty_c(B,[0,\infty))$ be a function satisfying $\int_B \rho dx =1$ and $\rho(-x)=\rho(x)$. Now for given $\varepsilon > 0$ we define 
		\[
		\rho_\varepsilon(x) = \frac{1}{\varepsilon^n}\rho \Bigl( \frac{x}{\varepsilon} \Bigr).
		\]
		Then $\rho_\varepsilon \in C^\infty_c(B_\varepsilon,[0,\infty))$ and for a given function $u \in L^1_\textnormal{loc}(\R^n)$ we define the $\varepsilon$-regularization of $u$ as the convolution between $u$ and $\rho_\varepsilon$, that is 
		\[
		u_\varepsilon(x) = (u \ast \rho_\varepsilon)(x) = \int_{\R^n} \rho_\varepsilon(x-y)u(y) \dd y.
		\]
		Now let $u_\varepsilon = \chi_F \ast \rho_\varepsilon \in C^\infty(\R^n)$, where $\chi_F$ is the characteristic function of $F$. It is not hard to show that $(\chi_F \ast \rho_\varepsilon) \rightarrow\chi_F$ in $L^1_{loc}(\R^n)$ as $\varepsilon \rightarrow0$. If $T \in C^1_c(\R^n,\R^n)$, then we have $\int_{\R^n} u_\varepsilon \textnormal{div }T = \int_F \textnormal{div}(T \ast \rho_\varepsilon)$. Indeed, by symmetry of $\rho_\varepsilon$ and Fubini's theorem\footnote{See Theorem 5.2.2 in \cite{cohn2013measure}} we get
		\begin{align*}
			\int_{\R^n} u_\varepsilon \divv T &= 
			\int_{\R^n} \Bigl( \int_{\R^n} \rho_\varepsilon(x-y)u(y)\dd y \Bigr)\: \divv T(x) \dd x \\ &=
			\int_{\R^n} \Bigl(\int_{\R^n} \divv T(x)\rho_\varepsilon(y-x) \dd x \Bigr)\: u(y) \dd y \\ &=
			\int_F \divv(T \ast \rho_\varepsilon).
		\end{align*}
		As $T$ is compactly supported, we have $\int_{\R^n} \textnormal{div}(u_\varepsilon T) = 0$, so that by integration by parts it follows that
		\begin{equation} \label{eq:halfspace1}
			- \int_{\R^n} \nabla u_\varepsilon \cdot T 
			=
			\int_F \textnormal{div}(T \ast \rho_\varepsilon) 
			= 
			\int_{\R^n}(T \ast \rho_\varepsilon) \cdot \mu_F 
			= 
			\int_{\partial^\ast F} (T \ast \rho_\varepsilon) \cdot \nu \dd |\mu_F|. 
		\end{equation}
		If $T = \varphi \nu'$ for $\varphi \in C^1_c(\R^n, \R^n)$ and $\nu'$ is a unit vector orthogonal to $\nu$, then \eqref{eq:halfspace1} gives 
		\[
		-\int_\R^n \frac{\partial u_\varepsilon}{\partial \nu'} \varphi = 0 \quad \forall  \varphi \in C^1_c(\R^n, \R^n),
		\]
		which implies $\dfrac{\partial u_\varepsilon}{\partial \nu'} \equiv 0$. If instead we choose $T = \varphi \nu$, $\varphi \in C^1_c(\R^n, \R^n)$, then applying \eqref{eq:halfspace1} again we get
		\[
		-\int_{\R^n} \frac{\partial u_\varepsilon}{\partial \nu} \varphi = \int_{\partial^\ast F} \varphi_\varepsilon \dd |\mu_F|.
		\]
		Note that the right-hand side is non-negative for every $\varphi \geq 0$. Since $\varphi \in C^1_c(\R^n, \R^n)$ is arbitrary and $u_\varepsilon$ is smooth, this implies that $\dfrac{\partial u_\varepsilon}{\partial \nu} \leq 0$ on $\R^n$. Together with $\dfrac{\partial u_\varepsilon}{\partial \nu'} \equiv 0$ we can therefore conclude that $u_\varepsilon(y)$ is a decreasing function of the variable $(y \cdot \nu)$, that is, there exists a decreasing function $f_\varepsilon \in C^\infty(\R;[0,1])$ such that
		\[
		(\chi_F \ast \rho_\varepsilon)(y)=u_\varepsilon(y)=f_\varepsilon(y \cdot \nu) \quad \forall y \in \R^n.
		\]    
		Letting $\varepsilon \rightarrow0$ we find $\lim_{\varepsilon \rightarrow0} f_\varepsilon(t) \in \{0,1\}$ for a.e. $t \in \R$. Hence there exists $a \in \R$ such that $f_\varepsilon(t) \rightarrow\chi_{(-\infty, a)}(t)$ for a.e. $t \in \R$, or 
		\[
		(\chi_F \ast \rho_\varepsilon)(y) = u_\varepsilon(y) \rightarrow\chi_{(-\infty,a)}(y \cdot \nu) \quad \textnormal{for a.e. } y \in \R^n.
		\]
		The result now follows as $(\chi_F \ast \rho_\varepsilon) \rightarrow\chi_F$ in $L^1_{loc}(\R^n).$
	\end{proof}

	\subsection{$Q_\pom$ for smooth boundaries}
	As a last step before discussing our results when it comes to values of $q_\pom(x)$ when $x \in \pom$ is a singularity on the boundary of $\Omega$, we show that when there are no singularities to be considered, then $q_\pom = 1$ everywhere and thus $Q_\pom = 1$. This was proved by Giusti in \cite{giusti1976boundary}. For the reader's convenience, we report here the proof.
	
	\begin{theorem} \label{thm:Giusti}
		Let $\pom$ be $C^1$ in a neighbourhood of $x_0 \in \pom$. Then, $q_{\pom}(x_0) = 1$.
	\end{theorem}
	
	\begin{proof}
		We can assume that $x_0=0$ and that $\pom$ can be represented as the graph of a $C^1$ function $f(x')$ for $x'=(x_1, \cdots, x_{n-1})$ such that $f(0)=0$ and $\nabla f(0) = 0$ and
		\[
		x_n > f(x') \quad \textnormal{ for } x \in \Omega \cap B(0,r) 
		\]
		Let $E \subset \Omega$ and let $\pi_E$ be the projection of  $\partial^\ast E \cap \pom$ on the hyperplane $x_n=0$. Then we have 
		\begin{equation*}
			\int_{\pom} \Tr(\ind_E) \: \dd \Hd^{n-1} = \Hd^{n-1}(\partial^\ast E \cap \pom) = \int_{\pi_E} \sqrt{1+|\nabla f(x')|^2} \dd x'
		\end{equation*}
		Furthermore setting $M_r \coloneqq \sup \{|\nabla f|(x), |x'| < r \}$ we get $M_r \rightarrow 0$ as $r \rightarrow 0$ and 
		\begin{equation} \label{eq:C1proof1}
			\int_\pom \Tr(\ind_E) \dd \Hd^{n-1} \leq
			\Bigl(\sqrt{1+M_r^2}\Bigr)\Hd^{n-1}(\pi_E)
			\leq
			(1+M_r)\Hd^{n-1}(\pi_E)
		\end{equation} 
		Now, let $T  \equiv e_n = (0, \cdots, 0, 1) \in \contvf$ be a constant vector field. Then by Theorem \ref{thm:Giorgo} we have
		\[
		0 = \int_E \divv T = \int_{\partial^\ast E \cap \Omega} T \nu_E \Hd^{n-1} + \int_{\partial^\ast E \cap \partial \Omega} T \nu_E \Hd^{n-1},  
		\]
		therefore  we get 
		\[
		- \int_{\partial^\ast E \cap \partial \Omega} T \cdot \nu_E \Hd^{n-1} = \int_{\partial^\ast E \cap \Omega} T \cdot \nu_E \Hd^{n-1} \leq \mu_E(\Omega) = P(E;\Omega). 
		\]
		On the other hand, 
		\begin{align*}
			-\int_{\partial^\ast E \cap \partial \Omega} T \cdot \nu_E \Hd^{n-1} &= -\int_{\pi_E} T \cdot \nu_E |\text{det } Df(x)| \dd x \\
			&= -\int_{\pi_E} T \cdot \frac{(\nabla f, -1)}{\sqrt{1+|\nabla f|^2}} \sqrt{1+|\nabla f|^2} \dd x \\
			&= - \int_{\pi_E} -1 \dd x \\
			&= \Hd^{n-1}(\pi_E).
		\end{align*}
		As such we find
		\begin{equation} \label{eq:C1proof2}
			P(E;\Omega) \geq \Hd^{n-1}(\pi_E)
		\end{equation}
		Combining \eqref{eq:C1proof1} and \eqref{eq:C1proof2} we get
		\begin{equation}
			\frac{1}{P(E;\Omega)} \int_\pom \Tr(\ind_E)(x) \dd\Hd^{n-1} \leq (1+M_r)
		\end{equation}
		If we now let $r \rightarrow 0$ the result follows.
	\end{proof}

	

	\section{The maximization problem} 
	Now that we have established the required mathematical language and preliminary results we can show why existence of a solution can't be guaranteed and formulate optimal constants for a specific class of singularities. In order to abbreviate notation, we will define the following symbols.
	\begin{definition} \label{def:alpha}
		Let $\C$ be a cone. We define $\alpha$ as
		\[ 
		\alpha(\C) \coloneqq \sup \Bigl\{ \frac{1}{P(E;\C)} \int_{\partial \C} \Tr(\ind_E)(x) \dd \Hd^{n-1} : |E| > 0, E \subseteq \C \text{ of finite perimeter} \Bigr\}.
		\]
	\end{definition}
	Note that for some point $x \in \pom$ that is locally described by a cone $\mathcal{C}$, then $q_\pom(x) = \alpha(\mathcal{C})$.
	Moreover, for any set $E \subseteq \C$ with positive volume, we write
	\[
	\beta(E) \coloneqq  \frac{1}{P(E;\C)} \int_{\partial \C} \Tr(\ind_E)(x) \dd\Hd^{n-1}.
	\]
	As such, $\alpha(\mathcal{C})$ now corresponds to the supremum of values that $\beta(E)$ can attain for admissible sets $E$ in a given cone $\C$.
	
	\subsection{Uncertainty of existence of a solution}
	While the value $\alpha(\mathcal{C})$ always exists, it is not necessarily true that there exists a solution $E$ to the maximization problem. As such, we will investigate this issue in this section first. We start off be remarking that if a solution exists, it will not be unique. Indeed let $E \subseteq \C$ be any set inside the cone and $\lambda > 0$. Then we have 
	\[
	\int_{\partial \C} \ind_{\lambda E}(x) \dd\Hd^{n-1} = \int_{\lambda^{n-1}\partial \C} \ind_{E}(x) \dd(\lambda^{n-1}\Hd^{n-1}) = \lambda^{n-1}\int_{\partial \C} \ind_{E}(x) \dd \Hd^{n-1},
	\]
	as follows clearly from the definition of $\Hd^{n-1}$ and the fact that $\C$ is scaling invariant. Similarly, $$P(\lambda E;\C) = \lambda^{n-1}P(E;\C).$$
	As such, it follows that $\beta(\lambda E) = \beta(E)$, thus any scalar multiple of an optimal set will also be optimal. 
	\\

	As we discussed in Section \ref{subsec:Lsc}, the general strategy to prove existence of a minimizer is by constructing a sufficiently well-behaved minimizing sequence and taking a converging subsequence using Theorem \ref{ExistTh}. However, in our case, even though we can construct sequences that meet the conditions of Theorem \ref{ExistTh}, we can not guarantee that the limit $E$ meets our condition that $|E|>0$. Indeed, the scaling invariance of solutions makes it possible to uniformly bound the perimeters, but this scaling could lead to $|E_n|$ converging to 0. On the other hand, if we fix $|E_n| = 1$, we can not ensure that the sequence stays bounded. \\
	A direct example can be given by letting $\C$ be a 3-dimensional cone generated by 
	\[
	f: \mathbb{S}^1 \rightarrow (0,\infty], \quad \theta \mapsto \sqrt{1+\tan(\theta)},
	\]which describes a plane folded in a $90^\circ$ angle as you can see in Figure \ref{fig:bookcone}. 
	\begin{figure} 
		\centering
		\includegraphics[scale=0.8]{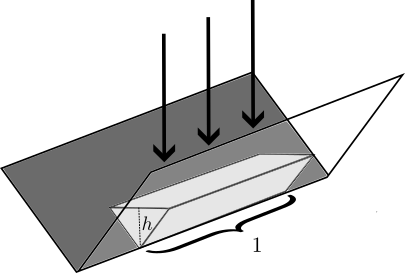}
		\caption{The 'book' cone }
		\label{fig:bookcone}
	\end{figure}
	Now suppose you let $E_n \subset \C$ be a prism as shown in the figure with height $h=\frac{1}{n}$. Then the perimeter inside is given as the sum of the area of the two triangles and the upper rectangle, thus
	\[
	P(E_n:\C) = 2 \cdot \Bigl( \frac{1}{2} \cdot h \cdot 2h \Bigr) + 2h \cdot 1 = 2h^2+2h = \frac{2}{n^2} + \frac{2}{n},
	\]
	and the perimeter on the boundary of $\C$ is then given by the sum of the two tilted rectangles
	\[
	\Enumerator = 2(\sqrt{2}h \cdot 1) = 2\sqrt{2}h = \frac{2\sqrt{2}}{n}.
	\]
	Consequently, the boundary will be
	\[
	\beta(E_n) = \frac{\frac{2\sqrt{2}}{n}}{\frac{2}{n^2} + \frac{2}{n}} 
	= \frac{2\sqrt{2}}{\frac{2}{n} + 2},
	\]
	and thus it follows that $\beta(E_n) \rightarrow \sqrt{2}$ as $\ninfty$. It is also clear that \eqref{eq:ExistTh_SupCond} and \eqref{eq:ExistTh_BallCond} are again satisfied, thus by \ref{ExistTh} a limit $E \in \R^n$ must exist. However since $|E_n| = h^2 = \frac{1}{n^2}$ we  have $|E|=0$, which is not admissible. 
	It will however turn out in Section \ref{sec:solution} that for this cone $\alpha(\C) = \sqrt{2}$ indeed holds and thus that $(E_n)$ indeed is a maximizing sequence, but that there simply does not exist a maximizer for this case.

	\subsection{Solution for bases with an inner described sphere} \label{sec:solution}
	
	In this section, we will provide a proof of the value of $\alpha(\C)$ together with its unique maximizers for cones where we have $G(f)= R \cdot \mathbb{S}^{n-1}$, that is, $f \equiv R$ for some constant $R>0$. Afterwards, we can generalize it to cones where the base $G(f)$ has an inscribed sphere.
	
	\begin{proposition} \label{CircleProof}
		Let $\C \subset \R^n$ be a cone such that the base is of the form $G(f) = R \cdot \mathbb{S}^{n-2}$ for some $R>0$, thus $\C = \bigl\{ (t \cdot B^{n-1}_R(0),t) \: : \: t \geq 0 \bigr\}$. Then, we have
		\[
		\alpha(\C)=\frac{\sqrt{1+R^2}}{R},
		\]
		Moreover, the sets of finite perimeter $E \subset \C$ such that $\beta(E)=\alpha(\C)$ are precisely the sets for which there exists an $t>0$ so that $E = H^-_t \cap \C$. 
		Here we use $H^-_t$ to denote the half-space $\{ x \in \R^n : x_n<t\}$.
	\end{proposition}

	\begin{proof} 
		\textit{Step 1:} Define a constant vector field $T \in \contvf$ by 
		$T \equiv (0,...,0,1)$. Then as $T$ is constant it has zero derivative in any direction so we have
		\begin{equation}
			\label{div0}
			\int_E \textnormal{div }T = 0.
		\end{equation}
		On the other hand, by Theorem \ref{thm:Giorgo}, we also have
		\begin{equation}\label{div1}
			\int_E \textnormal{div }T 
			= \int_{\partial^\ast E} T\cdot \nu_E \dd\Hd^{n-1} 
			= \int_{\partial^\ast E \cap \partial \C} T\cdot \nu_\C \dd\Hd^{n-1} + \int_{\partial^\ast E \cap \C} T\cdot \nu_E \dd\Hd^{n-1}.
		\end{equation}
		Indeed $\nu_E = \nu_\C$ on $\partial^\ast E \cap \partial \C$, as by \cite[Remark 12.4]{maggi2012sets} $\mu_E = \nu_C$ must hold there, thus their measure-theoretic normals must coincide $\Hd^{n-1}$-a.e..\\
		Putting \eqref{div0} and \eqref{div1} together, we get
		\begin{equation} \label{eq:DIV0}
			-\int_{\partial^\ast E \cap \partial \C} T\cdot \nu_\C \dd\Hd^{n-1} = \int_{\partial^\ast E \cap \C} T\cdot \nu_E \dd\Hd^{n-1}.
		\end{equation}
		\textit{Step 2:} We claim that 
		\begin{equation} \label{eq:TdotNuC}
			T \cdot \nu_C \equiv \frac{-R}{\sqrt{1+R^2}}.
		\end{equation}
		Indeed, every point $x \in \partial C$ can be written as 
		\[
		x = \bigl(y, \frac{1}{R}|y| \bigr),
		\]
		where $y = (y_1, ..., y_{n-1}) \in \R^{n-1}$ and for any fixed height $x_n = t$ all $y$ satisfy $|y|=tR$. Since at any point on a sphere the normal vector lies on the line through the point and the origin, it follows that 
		\[
		\nu_\C(x)=\nu_\C \biggl( \Bigl(   y, \frac{1}{R}|y| \Bigr)  \biggr) =\frac{(y,a)}{|(y,a)|}
		\]
		for such an $a$ that it is also orthogonal in the direction of $x_n$. As $\partial \C$ consists of straight lines starting at the origin, the vector $\vec{x}$ is itself tangent to $\C$ at $x$. We thus want $\nu_\C(x) \cdot x = 0$, which leads us to 
		\[
		0 = (y,a)\cdot \Bigl(   y, \frac{1}{R}|y| \Bigr) 
		= y \cdot y + \frac{a}{R}|y|
		= |y|^2 + \frac{a}{R}|y|
		\]
		As such we have $a = -R|y|$. It follows that
		\[
		|(y,a)|=|(y,-R|y|)=\sqrt{|y|^2 + R^2|y|^2}=|y|\sqrt{1+R^2}.
		\]
		Thus, the $n$-th component of $\nu_\C(x)$ is given by
		\[
		\frac{a}{|(y,a)|}=\frac{-R|y|}{|y|\sqrt{1+R^2}} = \frac{-R}{\sqrt{1+R^2}}.
		\]
		This proves that 
		\[
		T \cdot \nu_\C \equiv (0, \cdots, 0, 1) \cdot \nu_\C \equiv \frac{-R}{\sqrt{1+R^2}} 
		\]
		showing that \eqref{eq:TdotNuC} indeed holds. \\
		
		\textit{Step 3:} Substituting \eqref{eq:TdotNuC} into \eqref{eq:DIV0}, we get
		\begin{equation}
			\label{div2}
			\frac{R}{\sqrt{1+R^2}} \int_{\partial \C} \Tr(\ind_E) \dd\Hd^{n-1} = \int_{\partial^\ast E \cap \C} T\cdot \nu_E \dd\Hd^{n-1}.
		\end{equation}
		As we have $|T \cdot \nu_E| \leq 1$ by Cauchy-Schwarz, it follows that 
		\begin{equation}
			\label{div3}
			\int_{\partial^\ast E \cap \C} T\cdot \nu_E \dd\Hd^{n-1} \leq \int_{\partial^\ast E \cap \C} 1 \: \dd\Hd^{n-1} = P(E;\C).
		\end{equation}
		By \eqref{div2} and \eqref{div3} it follows that
		\begin{equation}
			\frac{1}{P(E;\C)}\int_{\partial \C} \Tr(\ind_E) \dd\Hd^{n-1}  \leq \frac{\sqrt{1+R^2}}{R}.
		\end{equation}
		Consequently, we have 
		\[ 
		\alpha(\C) \leq \frac{\sqrt{1+R^2}}{R}.
		\]
		Note that equality holds if there exists a set $E \subset \C$ such that equality holds in \eqref{div3}. Such a maximizing set must satisfy $T \cdot \nu_E = 1$ for $\Hd^{n-1}$-a.e. inside the cone $\C$, or equivalently, $\nu_E(x)=(0, ..., 0, 1)$ for $\Hd^{n-1}$-a.e. $x \in \partial E \cap \C$. 
		By Proposition \ref{halfspace} it follows then that 
		\begin{align*}
			E \cap \C &= \{ z \in \R^n : z \cdot (0,...,0,1) < t \} \cap \C = \{ z \in \R^n : z_n < t \} \cap \C
		\end{align*}
		for some $t \in \R$, and as $E \subset \C$ we must have $t>0$. 
		In conclusion, 
		\[
		\alpha(\C)=\frac{\sqrt{1+R^2}}{R}
		\]
		and the maximizing sets are exactly those sets $E \subset \C$ such that there exists an $t>0$ so that $E = H^-_t \cap \C$.
	\end{proof}
	
	For proving the previous Proposition is was essential that $(0,...,0,1)\cdot \nu_{\partial \C}$ was constant on $\partial \C$. One may thus wonder if we can generalize the statement to any cone for which this holds. For example, if $\C \subset \R^3$ and the base of the cone is a square, then the inner product $(0,...,0,1)\cdot \nu_{\partial \C}$ is also constant and equal to the product for the cone generated by the inner circle of the square. 
	
	\begin{theorem} \label{thm:Shapeswithinnercircle}
		Let $\C$ be a cone with a base $G(f) \subset \R^{n-1}$, such that there exists a sphere 
		\[
		S_R(\bar{x}) \coloneqq \{ y \in \R^{n-1} : \|x-y\| = R \} \subset \R^{n-1},
		\]
		such that any hyperplane $T_y \subset \R^{n-2}$ tangent to $G(f)$ at $y \in G(f)$ is also tangent to $S_R(x)$. Then we have 
		\[
		\alpha(\C)=\frac{\sqrt{1+R^2}}{R}.
		\]
		Moreover, a set $E\subset\C$ is a solution to the maximization problem defining $q_\pom(0)$ if and only if it is the intersection of the half-space $\{x\in\R^n : x\cdot \bar{x} \leq 0\}$ with $\C$.
	\end{theorem}
	
	\begin{proof}
		Without loss of generality, we assume that indeed $\bar{x}=0$ as we can rotate the cone, and write $\C_0$ for the cone generated by the inner circle $S_R(0)$. 
		
		\textit{Step 1:}
		First, we will show that the conditions on $G(f)$ are sufficient such that 
		\[
		(0,...,0,1) \cdot \nu_\C \equiv \frac{-R}{\sqrt{1+R^2}}.
		\]
		Let $y \in G(f)$ and let $T_y$ be tangent to $G(f)$ at $y$. Since we have $\partial \C = \{(t\,G(f),t) : t>0\}$, we define $\hat{T}_y = \{ (t \, T_y, t) : t>0 \} \subset \R^n$. As this construction linearly expands $T_y$ in the $n$-th axis it follows that $\hat{T}_y$ is a $(n-1)$-dimensional hyperplane and  for every $t>0$, $(t\, T_y, t)$ is tangent to $(t \, G(f), t)$, thus we may conclude that $\hat{T}_y$ is tangent to $\C$ on the line $\{(ty,t) \: : \: t>0 \} \in \partial \C$. 
		However, as $T_y$ is also tangent to $S_R(0)$, it follows is in the same way that for all $t>0$ $(t \, T_y, t)$ is tangent to $(t \, S_R(0), t)$ and thus that $\hat{T}$ is tangent to $\C_0$. This last statement implies that for all $t>0$ and $y \in G(f)$
		\begin{equation} \label{eq:TangCircleProduct}
			(0,...,0,1) \cdot \nu_\C((ty,t)) = (0,...,0,1) \cdot \nu_{\hat{T}_y} = (0,...,0,1) \cdot \nu_{\C_0}(y_0) = \frac{-R}{\sqrt{1+R^2}},
		\end{equation}
		where $y_0 \in \hat{T}_y \cap \C_0$. Since $t > 0$ and $y \in \C_0$ are arbitrary and any point on $\partial \C$ can be written in this form $(ty,t)$, it follows that
		\[
		(0,...,0,1) \cdot \nu_\C = \frac{-R}{\sqrt{1+R^2}}
		\]
		holds in general. \vspace{5mm} \\
		
		\textit{Step 2:} 
		Let $T \in \contvf$ be $T \equiv (0, ...,0,1)$, then thanks to Step 1 and using similar computations as in the proof of Theorem \ref{CircleProof} we get,
		\[
		\int_{\partial^\ast E \cap \C} T \cdot \nu_E \dd\Hd^{n-1} =
		\frac{-R}{\sqrt{1+R^2}} \int_{\partial \C} \Tr(\ind_E) \dd\Hd^{n-1}. 
		\]
		By \eqref{div0} and \eqref{div1} we can again conclude 
		\[
		\frac{R}{\sqrt{1+R^2}} \int_{\partial \C} \Tr(\ind_E) \dd\Hd^{n-1} = \int_{\partial^\ast E \cap \C} T\cdot \nu_E \dd\Hd^{n-1}.
		\]
		Thus by \eqref{div3} we have
		\[
		\alpha(\C) = \frac{1}{P(E;\C)}\int_{\partial \C} \Tr(\ind_E) \dd\Hd^{n-1}  \leq \frac{\sqrt{1+R^2}}{R}.
		\]
		
		Finally, arguing as in the proof of Theorem \ref{CircleProof} it follows that the maximizing sets $E$ such equality holds are exactly those sets $E \subset \C$ such that there exists an $t > 0$ so that $E = H^-_t \cap \C$.
	\end{proof} 
	
	\begin{figure}[H]
		\centering
		\includegraphics[scale=0.7]{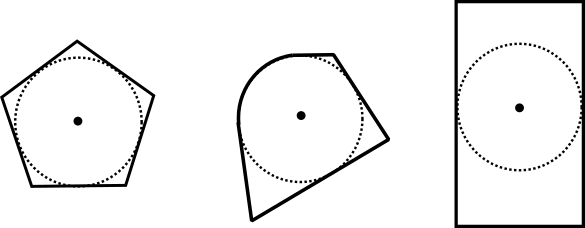}
		\caption{Even though the first two shapes for $G(f)$ satisfy the conditions of Theorem \ref{thm:Shapeswithinnercircle}, the third does not.}
		\label{fig:InnerCircle}
	\end{figure}
	
	Figure \ref{fig:InnerCircle} gives some intuition about the class of shapes Theorem \ref{thm:Shapeswithinnercircle} allows in the case of 3-dimensional cones. Indeed, every convex regular polyhedron has an inner described circle, but also more irregular shapes as the second one are possible. On the other hand, non-square rectangles like the third figure do not have an inscribed circle and therefore the theorem does not apply.
	
	\begin{remark}
		Note that in the case where the cone $\C \subset \R^2$ is 2-dimensional, $G(f)$ will exist of only two points, as the domain of $f$ is then $\mathbb{S}^0=\{1,-1\}$. As such, simply taking 
		\[
		x=\frac{1}{2}\bigl(f(1)+f(-1)\bigr) \textnormal{  and  } R=\frac{1}{2}\bigl|f(1)-f(-1)\bigr|
		\]
		it follows that Theorem \ref{thm:Shapeswithinnercircle} applies to any choice of $f$. Consequently, in 2 dimensions the minimizer always has a top surface orthogonal to the bisector.
	\end{remark}
	
	\begin{remark}
		Giusti also mentioned that the solution for 2 dimensions would be as provided by Theorem~\ref{thm:Shapeswithinnercircle} and the last Remark on page 15 of \cite{giusti1976boundary}, but did not provide a proof.
	\end{remark}
	
	\begin{remark}
		One can now indeed see that the sequence of sets $(E_n)$ constructed for the book-cone example in Figure \ref{fig:bookcone} was actually maximizing. Indeed, in that case, the shape $G(f)$ existed of 2 parallel lines at a distance 2 from each other, thus admitting a inner circle of radius $R=1$, which by Theorem \ref{thm:Shapeswithinnercircle} implies that the optimal ratio must be $\sqrt{2}$.
	\end{remark}
	The strategy of the previous proof can also be adapted to get a bound for more general classes of cones.
	
	\begin{corollary} \label{cor:BoundEstimate}
		Let a cone $\C$ have a base $G(f) \subset \R^{n-1}$, such that there exists a sphere 
		\[
		S_R(\bar{x}) \coloneqq \{ y \in \R^{n-1} : \|\bar{x}-y\| = R \},
		\]
		such that any $(n-2)$-dimensional plane $T_y \subset \R^{n-1}$ tangent to $G(f)$ at $y \in G(f)$ is either tangent to $S_R(x)$ or $T_y \cap S_R(x) = \emptyset$, then
		\[
		\alpha(\C) \leq \frac{\sqrt{1+R^2}}{R}.
		\]
	\end{corollary}
	
	\begin{proof}
		The proof of this Corollary is a result of a slight alteration of the proof of Theorem \ref{thm:Shapeswithinnercircle}. Just as before, first re-orientate the cone such that $x$ corresponds to $(0,\cdots,0,1)$ in $\C = \{ (tG(f),t), t>0 \}$. 
		We will again use the constant vector field $T \equiv (0,\cdots,0,1)$. Denote $\C_\rho$ for the cone $\C_\rho = \{(t \rho \, \mathbb{S}^{n-2},t) \: : \: t>0 \} $.
		By equation \eqref{eq:TangCircleProduct} we know that the inner product $T \cdot \nu_\C(y_0)$ for some point $y_0 \in \partial \C$ is determined by which cone $\C_\rho$ the tangent plane $\hat{T}_{y_0}$ is tangent to. \\
		Because of the condition on the circle $S_R(0)$, for all $y_0 \in \partial \C$, the cone $\C_\rho$ where $\hat{T}_{y_0}$ is tangent to will have radius $\rho \geq R$. For $y_0$ we then have
		\[
		T \cdot \nu_\C(y_0) = \frac{-\rho}{\sqrt{1+\rho^2}} \leq \frac{-R}{\sqrt{1+R^2}}.
		\]
		As such, we get
		\[
		\int_{\partial^\ast E \cap \C} T \cdot \nu_E \dd\Hd^{n-1} \leq
		\frac{-R}{\sqrt{1+R^2}} \int_{\partial \C} \Tr(\ind_E) \dd\Hd^{n-1}.
		\]
		By \eqref{div0} and \eqref{div1}, we have
		\[
		\frac{R}{\sqrt{1+R^2}} \int_{\partial \C} \Tr(\ind_E) \dd\Hd^{n-1} 
		\leq
		-\int_{\partial^\ast E \cap \partial \C} T \cdot \nu_C \dd\Hd^{n-1}
		=
		\int_{\partial^\ast E \cap \C} T\cdot \nu_E \dd\Hd^{n-1}.
		\]
		Now we apply \eqref{div3} as before and find 
		\[
		\frac{1}{P(E;\C)}\int_{\partial \C} \Tr(\ind_E) \dd\Hd^{n-1} \leq \frac{\sqrt{1+R^2}}{R}.
		\]
		In conclusion,
		\[
		\alpha(\C) \leq \frac{\sqrt{1+R^2}}{R}.
		\]
	\end{proof}
	As we have
	\[\frac{\sqrt{1+R^2}}{R} = \sqrt{\frac{1}{R^2}+1},\]
	this is a decreasing function so that one achieves the sharpest bound by finding the biggest circle $S_R(x)$ that satisfies the conditions in Corollary \ref{cor:BoundEstimate}.
	
	\vspace{5mm}
	
	\subsection{Outlook}
	Although the conditions of Theorem \ref{thm:Shapeswithinnercircle} admit a wide range of shapes, we already saw that in the 3-dimensional case, it fails even when the base is a regular shape like a non-square rectangle, even though it seems intuitive that the half-space solution is still optimal. On the other hand, if $G(f)$ has a very thin spike, one might actually expect the minimizer, if existing, to be found inside such a spike. A major question for further research thus might be to ask if one can relax the conditions of Theorem \ref{thm:Shapeswithinnercircle} to, for example, convex cones or some other condition.


	\section*{Acknowledgments}
	To start off I would like to thank my supervisor Prof. Cristoferi for guiding me through the writing of the thesis, and I would like to thank the Radboud Honours Academy for making it possible to improve my thesis
	by sponsoring a trip to the United States, where I could discuss my research problem with several
	experts. For that matter, I want to thank Prof. Leoni, Prof. Neumayer and Prof. Smit Vega Garcia
	for guiding me during my visit and for having several meetings with me that greatly improved
	my understanding of the subject. I also was very pleased by the discussions I had with the PhD
	candidates of Prof. Leoni, Wesley Caldwell and Francesco Cassandra, during my stay.

	\bibliographystyle{siam}
	\bibliography{references.bib}
	
\end{document}